\documentclass[11pt]{article}
\usepackage{amscd,amssymb,stmaryrd}

\usepackage{graphicx}
\usepackage{subcaption}
\usepackage[utf8]{inputenc}
\usepackage[export]{adjustbox}
\usepackage{wrapfig}

\usepackage{amsthm}

\usepackage{color}
\usepackage{mathtools}

\usepackage{hyperref}
\usepackage[english]{babel}
\usepackage{booktabs}

\usepackage{float}
\usepackage[linesnumbered,ruled,vlined, ruled,vlined]{algorithm2e}

\newtheorem{theorem}{Theorem}[section]

\newtheorem{lemma}[theorem]{Lemma}

\newtheorem{remark}{Remark}
\theoremstyle{definition}
\newtheorem{definition}[theorem]{Definition}

\usepackage{tikz}
\usepackage{pifont}%

\usepackage{diagbox}
\usetikzlibrary{decorations.pathreplacing,calligraphy}
\usepackage{pgfplots}
\pgfplotsset{compat=1.11}
\usepgfplotslibrary{groupplots}

\usetikzlibrary{shapes.misc}

\tikzset{cross/.style={cross out, draw=black, minimum size=2*(#1-\pgflinewidth), inner sep=0pt, outer sep=0pt},
cross/.default={1pt}}

\title{A discretization-invariant extension and analysis of some deep operator networks}

\author{ Zecheng Zhang\footnote{Department of Mathematics, Florida State University, Tallahassee, FL 32304, USA. (Email: zecheng.zhang.math@gmail.com)}, ~Wing Tat Leung \footnote{Department of Mathematics, City University Hong Kong, Hong Kong, China. (Email: wtleung27@cityu.edu.hk)}, 
~Hayden Schaeffer \footnote{Department of Mathematics, UCLA, Los Angeles, CA 90095. (Email: hayden@math.ucla.edu)} }

\begin{document}

\maketitle

\begin{abstract}

We present a generalized version of the discretization-invariant neural operator in \cite{zhang2022belnet} and prove that the network is a universal approximation in the operator sense. Moreover, by incorporating additional terms in the architecture, we establish a connection between this discretization-invariant neural operator network and those discussed in \cite{chen1995universal} and \cite{lu2021learning}. The discretization-invariance property of the operator network implies that different input functions can be sampled using various sensor locations within the same training and testing phases. Additionally, since the network learns a ``basis'' for the input and output function spaces, our approach enables the evaluation of input functions on different discretizations. To evaluate the performance of the proposed discretization-invariant neural operator, we focus on challenging examples from multiscale partial differential equations. Our experimental results indicate that the method achieves lower prediction errors compared to previous networks and benefits from its discretization-invariant property.
\end{abstract}

\section{Introduction}

Operator learning \cite{chen1995universal, lu2021learning, zhang2022belnet, li2020fourier} is an approach to approximate mappings between two function spaces, and can be seen as a generalization of the standard machine learning architectures. These approaches have gained significant attention in recent years, particularly for their applicability to scientific computing applications that require approximations of solution operators for spatiotemporal dynamics \cite{lu2021learning, zhang2022belnet, li2020fourier}. Operator networks have been used to approximate solutions to parametric partial differential equations (PDEs) \cite{lu2022comprehensive, zhang2022belnet, jin2022mionet, bhattacharya2021model}. Furthermore, operator learning has been employed for modeling control problems in dynamical systems \cite{li2022learning, michalowska2023neural}. Additionally, operator learning can be used to train models with varying levels of fidelity \cite{howard2022multifidelity, lu2022multifidelity, howard2023multifidelity} and for data-driven prediction \cite{pathak2022fourcastnet, cao2023deep}.

Operator learning was initially proposed in \cite{chen1995universal, chen1993approximations} using a shallow network architecture and was shown to be a universal approximation for nonlinear continuous operators. Building upon this work, the Deep Operator Neural Network (DON) was developed in \cite{lu2021learning, jin2022mionet} and extended the network in \cite{chen1995universal} to deep architectures. In particular, \cite{lu2021learning} extend the two layer operator network to networks of arbitrary depth, while \cite{jin2022mionet} generalized the operator network to handle multi-input and multi-output problems. Convergence analysis of DON can be found in \cite{lanthaler2022error, lanthaler2023curse}. Additionally, \cite{lin2021accelerated, lin2023b} investigated operator learning in the presence of noise and proposed accelerated training methodologies for DON. Another approach, known as the Fourier neural operator (FNO), was introduced and analyzed in \cite{li2020fourier, kovachki2021universal, wen2022u, zhu2023fourier}. FNO uses the Fourier transformation and inverse Fourier transformation on the kernel integral approximation for approximating an operator. A comparison between DON and FNO in terms of theory and computational accuracy is presented in \cite{lu2022comprehensive}. Additional noteworthy operator learning frameworks include \cite{zhang2022belnet, o2022derivate, qian2022reduced, zhu2023fourier}.

The choice of discretizations and domains is crucial in operator learning, as both the input and output are functions. A neural operator is said to be input (output) discretization-invariant if the network can handle varying discretizations for the input (output) function. This means that the input (or output) functions can be evaluated on different grids during both training and testing phases \cite{zhang2022belnet, li2020fourier, li2020neural}. This property is sometimes referred to as resolution-invariant or a non-uniform mesh approach. A discretization-invariant method is one that does not require the (1) the input discretization to be fixed, (2) the output discretization to be fixed, and (3) the input and output spaces to be the same \cite{zhang2022belnet}.
The Basis Enhanced Learning operator network (BelNet) was developed as a discretization-invariant approach. BelNet shares similarities with DON \cite{meuris2023machine} as it learns a representation for the output function space through the \textit{construction net} (see Figure \ref{fig_network_general}). However, BelNet also learns the projection of the input functions onto a set of ``basis'' terms obtained during training using the \textit{projection net}. The architecture of BelNet resembles an encoder-decoder, where the projection net acts as an encoder and the subsequent layers in the network decode the reduced-order model produced by the earlier layers. Numerical experiments presented in \cite{zhang2022belnet} demonstrate that BelNet can approximate nonlinear operators without relying on fixed grids, although this property has not been formally proven.

In this work, we present a generalization of BelNet and provide a proof of the universal approximation theorem in the operator sense. We introduce a sub-network structure called the \textit{nonlinear net}, which enhances the flexibility of the architecture. This nonlinear net is motivated by the universal approximation theorem proposed by \cite{chen1995universal}. Our proof strategy involves establishing connections between various discretizations through a proxy sampling that is fixed, thus allowing us to leverage existing approximation theorems. Furthermore, our proof introduces a more comprehensive approach through the encoder-decoder structure. Specifically, we demonstrate that our model obtains a reduced order model within a subnetwork, which while often stated in the literature, has not been shown. To distinguish our proposed network from previous work, we refer to the BelNet framework introduced in \cite{zhang2022belnet} as ``vanilla BelNet,'' while our new network is referred to as ``BelNet.'' The proposed network is related to the parallel work of \cite{hua2023basis}, who developed a similar structure based on a universal approximation result for operators with varying sensors. However, the analysis in \cite{hua2023basis} relies on having access to the continuous inner product layer, which does not imply that the conclusions hold for the discrete (approximated) network.

 Data plays a pivotal role in augmenting the learning process for physical systems. Notably, one can gain insights into the underlying principles governing physical phenomena through data-discovery \cite{schmidt2009distilling, brunton2016discovering, brunton2016sparse, schaeffer2013sparse, schaeffer2017learning, mangan2016inferring, schaeffer2018extracting, raissi2018hidden, schaeffer2017integral, schaeffer2020extracting2}. By employing data-driven approaches, these works have effectively extracted and learned valuable information about the intricate dynamics of the physical systems or governing model. This is one potential of data-driven methods for scientific enhancement and for modeling of complex physical processes. Recently, related methods for solving multiscale problems were proposed, wherein real observation data is employed to enhance a coarse-scale multiscale model \textcolor{black}{\cite{zhang2023homogenization}}. The approach involves training an operator that can map the coarse-scale solution (input function) to a finer-scale solution (output function) using the finer-scale solution obtained at specific locations within the domain.

To evaluate the performance of BelNet, we examine its effectiveness in solving the viscous Burgers' equation and learning the mapping between two multiscale models. Previous work \textcolor{black}{\cite{zhang2023homogenization}} demonstrated the concept of mapping between coarse and fine-scale solutions using operator learning, showing its applicability in various examples using DON. In our work, we show that BelNet offers more flexibility in selecting observation points for the coarse-scale input functions. To introduce additional complexity to the problem, we employ random sensors to sample the input functions. The corresponding section provides detailed results.

\subsection{Contributions}
We summarize the key contributions in this work below.
\begin{enumerate}
    \item We introduce a generalization of the vanilla BelNet \cite{zhang2022belnet}, a neural operator that is discretization invariant, and a proof of the universal approximation property for this extended model.
    \item The new BelNet extends the universal approximation results of \cite{chen1995universal, chen1993approximations, lu2021learning}.
    \item  We show a learning approach to map between two multiscale models on different grids and showcase the effectiveness of BelNet in handling challenging observation data.
\end{enumerate}

The rest of the paper is organized as follow. In Section \ref{sec_prelim}, we will review DON, vanilla BelNet, and the extended BelNet. We then present the universal approximation analysis in Section \ref{sec_proof}. In Section \ref{sec_numerical}, we present numerical experiments.

\section{Preliminary Results and Important Lemmata}
\label{sec_prelim}

Let $Y$ be a Banach space and assume that $K_1\subset Y$ and $K_2\subset \mathbb{R}$ are both compact.
Also, let $V\subset C(K_1)$ be compact and $G: V\rightarrow C(K_2)$ be a continuous and nonlinear operator. DON (see Figure \ref{fig_don_structure}) approximates the operator $G$ using a deep neural network. Specifically, 
 in \cite{chen1995universal} it was shown that  for any $\epsilon>0$, there exists positive integers $M, N, K$, constants $c_i^k, \zeta_k, \theta_i^k, \varepsilon_{ij}^k\in\mathbb{R}$, points $\omega_k\in\mathbb{R}^d$, $y_j\in K_1$, where $i \in [M]$, $k\in [K]$, $j \in [N]$ such that
    \begin{align*}
        \bigg| G(u)(x) - \sum_{k = 1}^K \sum_{i = 1}^M c_i^k\, g\left(\sum_{j = 1}^N\varepsilon_{ij}^ku(y_j)+\theta_i^k\right)\, g(\omega_k\cdot x+\zeta_k)\bigg|<\epsilon
    \end{align*}
    holds for all $u\in V$ and $x\in K_2$.

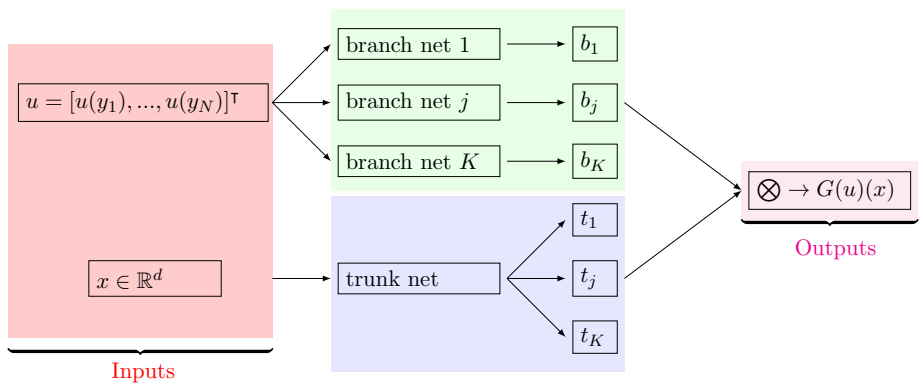
\begin{figure}[H]
\centering
\scalebox{.78}{\begin{tikzpicture}[scale = 1]
 \fill [green!10] (2, 4.5) rectangle (7, 7.6);
  \fill [blue!10] (2, 1.4) rectangle (7, 4.4);
    \fill [red!20] (-3.5, 2) rectangle (1, 7);

    \draw[ultra thick] [decorate,
    decoration = {calligraphic brace, mirror}] (-3.5, 1.8) --  (1, 1.8);
\node at (-1.2, 1.4) {\textcolor{red}{Inputs}};

\fill [magenta!10] (9, 4) rectangle (12, 5);
    \draw[ultra thick] [decorate,
    decoration = {calligraphic brace, mirror}] (9, 4) --  (12, 4);
\node at (10.6, 3.5) {\textcolor{magenta}{Outputs}};

\node[draw, text width=4cm] at (-1.2, 6) {$u = [u(y_1), ..., u(y_N)]^\intercal$};

 \draw [-latex ](1,6) -- (2, 7);
 \node[draw, text width = 2.5cm] at (3.5, 7) {branch net $1$};
 \draw [-latex ](5, 7) -- (6, 7);
 \node[draw, text width = 0.5cm] at (6.5, 7) {$b_1$};
 
 \draw [-latex ](1,6) -- (2, 5);
  \node[draw, text width = 2.5cm] at (3.5, 5) {branch net $K$};
  \draw [-latex ](5, 5) -- (6, 5);
  \node[draw, text width = 0.5cm] at (6.5, 5) {$b_K$};
  
  \draw [-latex ](1,6) -- (2, 6);
 \node[draw, text width=2.5cm] at (3.5, 6) {branch net $j$};
 \draw [-latex ](5, 6) -- (6, 6);
\node[draw, text width = 0.5cm] at (6.5, 6) {$b_j$};

 \node[draw, text width = 2cm] at (-1, 3) {$x\in\mathbb{R}^d$};
  \draw [-latex ](1, 3) -- (2, 3);
 \node[draw, text width = 2.5cm] at (3.5, 3) {trunk net};
 
 \draw [-latex ](5, 3) -- (6, 4);
\node[draw, text width = 0.5cm] at (6.5, 4) {$t_1$};

  \draw [-latex ](5, 3) -- (6, 3);
\node[draw, text width = 0.5cm] at (6.5, 3) {$t_j$};

   \draw [-latex ](5, 3) -- (6, 2);
 \node[draw, text width = 0.5cm] at (6.5, 2) {$t_K$};

\draw [-latex ](7, 6) -- (9, 4.5);
  \draw [-latex ](7, 3) -- (9, 4.5);

  \node[draw, text width = 2.5cm] at (10.5, 4.5) {$\bigotimes\rightarrow G(u)(x)$};

\end{tikzpicture}}
\caption{Stacked version DON. $\bigotimes$ denotes the inner product in $\mathbb{R}^K$.}
\label{fig_don_structure}
\end{figure}

The sensors are denoted by $y_i\in K_1$ and the input function $u$ is evaluated on the sensors. That is $y = [y_1, ..., y_N]^{\intercal}$ is the collection of points that represent the discrete grid for the input functions. Theorem 5 in \cite{chen1995universal} states that the sensors for all input functions $u$ must be the same (i.e. the input discretization must be fixed). This constraint imposes limitations on the applicability of the DON framework in the regime where one does not have control over the input functions or the sensor locations. Ideally, it would be desirable for the operator learning approach to allow for different input functions $u$ to have varying or non-uniform discretization, i.e. to have a discretization-invariant method.

The vanilla BelNet, as displayed in Figure \ref{fig_network_s1}, learns the basis functions for both the input and output function spaces. The authors provide an explanation of the network structure by examining a special case (linear) and validating the discretization-invariant property through various numerical experiments. Mathematically, let us introduce weights and biases, $q^{k}\in\mathbb{R}^{d}$, $W_y^{1, k}\in\mathbb{R}^{N_1\times N}$, $W_y^{2, k}\in\mathbb{R}^{N\times N_1}$, 
$b_x^k\in \mathbb{R}$, and $b_y^k\in\mathbb{R}^{N_1}$, where $k = 1, ..., K$, and activation functions $a_x$, $a_y$ and $a_u$, then the vanilla BelNet, denoted by $N_{\theta}$, approximates the operator $G$ as follows,
\begin{align}
    G(u)(x) \approx N_{\theta}(u(y), y)(x) = \sum_{k = 1}^K  a_x  \left((q^{k})^\intercal x+ b_x^k \right) \, a_u \left(\hat{u}^\intercal W_{y}^{2, k}\left(a_y(W_{y}^{1, k} y + b_y^k )  \right) \right),
    \label{intro_formulation3}
\end{align}
for $x\in K_2\subset \mathbb{R}$, $u\in V$,  and where $y = [y_1,...,y_N]^{\intercal}\subset K_1^N$ and $\hat{u} = [u(y_1), ..., u(y_N)]^\intercal$.
The network structure is also displayed in Figure \ref{fig_network_s1}. We do not assume that the sensors $y_i\in K_1$ are uniform for all input functions (i.e. they are not fixed). 
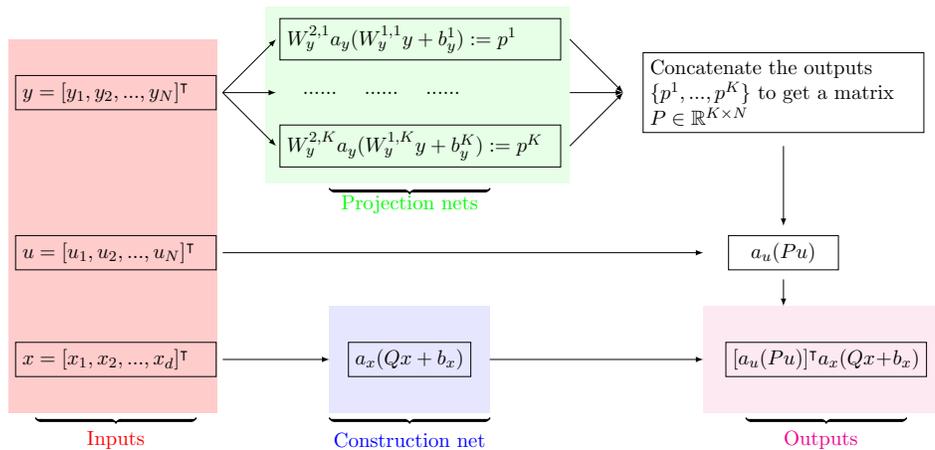
\begin{figure}[H]
\centering
\scalebox{.71}{\begin{tikzpicture}[scale = 1]

\fill [red!20] (-3, 0) rectangle (0.9, 7);
\draw[ultra thick] [decorate,
    decoration = {calligraphic brace, mirror}] (-2.5, -0.1) --  (0.5, -0.1);
\node at (-1, -0.5) {\textcolor{red}{Inputs}};

\fill [green!10] (1.8, 4.3) rectangle (7.5, 7.6);
\draw[ultra thick] [decorate,
    decoration = {calligraphic brace, mirror}] (3, 4.2) --  (6, 4.2);
\node at (4.5, 3.9) {\textcolor{green}{Projection nets}};

\fill [blue!10] (3, 0) rectangle (6, 2);

\draw[ultra thick] [decorate,
    decoration = {calligraphic brace, mirror}] (3, -0.1) --  (6, -0.1);
\node at (4.5, -0.5) {\textcolor{blue}{Construction net}};

\fill [magenta!10] (10., 0) rectangle (14.5, 2);

\draw[ultra thick] [decorate,
    decoration = {calligraphic brace, mirror}] (10.8, -0.1) --  (13.8, -0.1);
\node at (12.2, -0.5) {\textcolor{magenta}{Outputs}};
    
 \node[draw, text width = 3.5cm] at (-1, 6) {$y = [y_1, y_2, ..., y_N]^\intercal$};
 
 \draw [-latex ](1,6) -- (2, 7);
 \node[draw, text width = 5cm] at (4.7, 7) {$W_y^{2, 1}a_y(W_y^{1,1}y+b_y^1):=p^1$};
 
 \draw [-latex ](1,6) -- (2, 5);
  \node[draw, text width = 5cm] at (4.7, 5) {$W_y^{2, K}a_y(W_y^{1, K}y+b_y^K):= p^K$};
  
  \draw [-latex ](1,6) -- (2, 6);
 \node[text width=5cm] at (5, 6) {$... ...$ \quad $... ...$ \quad $... ...$};

 \draw [-latex ](7.5, 7) -- (8.5, 6);
 \draw [-latex ](7.5, 6) -- (8.5, 6);
 \draw [-latex ](7.5, 5) -- (8.5, 6);

 \node[draw, text width = 5cm] at (11.5, 6) {Concatenate the outputs $\{p^1, ..., p^K\}$ to get a matrix $P \in\mathbb{R}^{K\times N}$};

\node[draw, text width=3.5cm] at (-1, 3) {$u = [u_1, u_2, ..., u_N]^\intercal$};
\draw [-latex ](1, 3) -- (10, 3);

\node[draw, text width = 1.8 cm, align = center] at (11.5, 3) {$a_u(Pu)$};

\draw [-latex ](11.5, 5) -- (11.5, 3.5);

 \node[draw, text width = 3.5cm] at (-1, 1) {$x = [x_1, x_2, ..., x_d]^\intercal$};
 \draw [-latex ](1,1) -- (3, 1);
\node[draw] at (4.5, 1) {$a_x(Qx + b_x)$};
 \draw [-latex ](6,1) -- (10, 1);

\node[draw, text width = 3.5 cm, align = center] at (12.3, 1) {$[a_u(Pu)]^\intercal a_x(Qx + b_x)$};

\draw [-latex ](11.5, 2.5) -- (11.5, 2.0);

\end{tikzpicture}}
\caption{Plot of the vanilla BelNet structure. Projection nets are $K$ independent fully connected neural network with weights and bias $W_y^{2, k}\in\mathbb{R}^{N\times N_1}$, $W_y^{1, k}\in\mathbb{R}^{N_1\times N}$ and $b_y^k\in\mathbb{R}^{N_1}$. Construction net is a fully connected neural network with weights and bias $Q\in\mathbb{R}^{K\times d}$ and $b_x\in\mathbb{R}^d$. Here $Q = [q^1, q^2, ..., q^K]$, where $q^i\in\mathbb{R}^d$ are defined in Equation (\ref{intro_formulation3}). In addition, $a_x, a_y, a_u$ are activation functions. }
\label{fig_network_s1}
\end{figure}

The motivation behind the vanilla BelNet stems from Mercer's theorem, which involves approximating the linear operator through a kernel integral formulations \cite{zhang2022belnet}. However, in order to introduce nonlinearity, the activation function $a_u$ is incorporated. For flexibility and expressiveness, the authors in \cite{zhang2022belnet} included an extra trainable layer prior to the activation function, denoted by the term $a_u(WPu)$. Here, $W$ represents a trainable matrix of the appropriate dimension.

In this work, we generalize the vanilla BelNet and prove the universal approximation theorem. 
Instead of applying an activation function, we design a network to enforce the nonlinearity, see Figure \ref{fig_network_general}. The additional subnetwork is theoretically consistent with an operator approximation and is partly motivated through the analysis detailed in Section~\ref{sec_proof}.

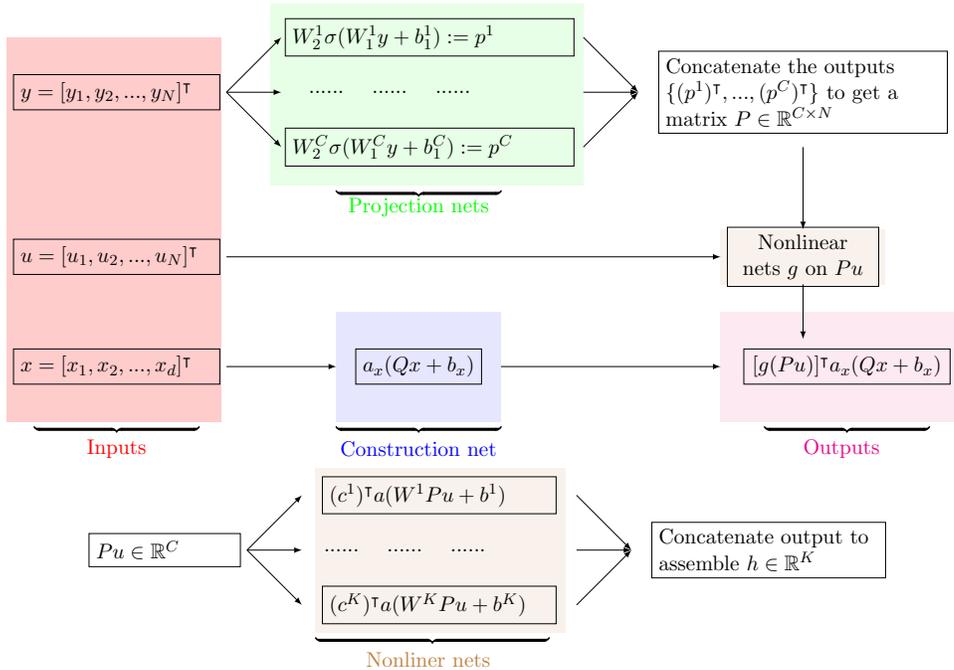
\begin{figure}[H]
\centering
\scalebox{.73}{\begin{tikzpicture}[scale = 1]

\fill [red!20] (-3, 0) rectangle (0.9, 7);
\draw[ultra thick] [decorate,
    decoration = {calligraphic brace, mirror}] (-2.5, -0.1) --  (0.5, -0.1);
\node at (-1, -0.5) {\textcolor{red}{Inputs}};

\fill [green!10] (1.8, 4.3) rectangle (7.5, 7.6);
\draw[ultra thick] [decorate,
    decoration = {calligraphic brace, mirror}] (3, 4.2) --  (6, 4.2);
\node at (4.5, 3.9) {\textcolor{green}{Projection nets}};

\fill [blue!10] (3, 0) rectangle (6, 2);

\draw[ultra thick] [decorate,
    decoration = {calligraphic brace, mirror}] (3, -0.1) --  (6, -0.1);
\node at (4.5, -0.5) {\textcolor{blue}{Construction net}};

\fill [magenta!10] (10., 0) rectangle (14.5, 2);

\fill [brown!10] (10, 2.5) rectangle (13, 3.5);

\draw[ultra thick] [decorate,
    decoration = {calligraphic brace, mirror}] (10.8, -0.1) --  (13.8, -0.1);
\node at (12.2, -0.5) {\textcolor{magenta}{Outputs}};
    
 \node[draw, text width = 3.5cm] at (-1, 6) {$y = [y_1, y_2, ..., y_N]^\intercal$};
 
 \draw [-latex ](1,6) -- (2, 7);
 \node[draw, text width = 5cm] at (4.7, 7) {$W_2^{1}\sigma(W_1^{1}y+b_1^1):=p^1$};
 
 \draw [-latex ](1,6) -- (2, 5);
  \node[draw, text width = 5cm] at (4.7, 5) {$W_2^{C}\sigma(W_1^{C}y+b_1^C):= p^C$};
  
  \draw [-latex ](1,6) -- (2, 6);
 \node[text width=5cm] at (5, 6) {$... ...$ \quad $... ...$ \quad $... ...$};

 \draw [-latex ](7.5, 7) -- (8.5, 6);
 \draw [-latex ](7.5, 6) -- (8.5, 6);
 \draw [-latex ](7.5, 5) -- (8.5, 6);

 \node[draw, text width = 5cm] at (11.5, 6) {Concatenate the outputs $\{(p^1)^{\intercal}, ..., (p^C)^{\intercal}\}$ to get a matrix $P \in\mathbb{R}^{C\times N}$};

\node[draw, text width=3.5cm] at (-1, 3) {$u = [u_1, u_2, ..., u_N]^\intercal$};
\draw [-latex ](1, 3) -- (10, 3);

\node[draw, text width = 2.5 cm, align = center] at (11.5, 3) {Nonlinear nets $g$ on $Pu$};

\draw [-latex ](11.5, 5) -- (11.5, 3.5);

 \node[draw, text width = 3.5cm] at (-1, 1) {$x = [x_1, x_2, ..., x_d]^\intercal$};
 \draw [-latex ](1,1) -- (3, 1);
\node[draw] at (4.5, 1) {$a_x(Qx + b_x)$};
 \draw [-latex ](6,1) -- (10, 1);

\node[draw, text width = 3.5 cm, align = center] at (12.3, 1) {$[g(Pu)]^\intercal a_x(Qx + b_x)$};

\draw [-latex ](11.5, 2.5) -- (11.5, 1.5);

\end{tikzpicture}}
\scalebox{.73}{\begin{tikzpicture}[scale = 1]

\fill [brown!10] (1.75, 4.5) rectangle (6.3, 7.5);
\draw[ultra thick] [decorate,
    decoration = {calligraphic brace, mirror}] (1.75, 4.4) --  (6, 4.4);
\node at (3.8, 4.0) {\textcolor{brown}{Nonliner nets}};

 \node[draw, text width = 2.5cm] at (-1, 6) {$Pu\in\mathbb{R}^C$};
 
 \draw [-latex ](0.5, 6) -- (1.5, 7);
  \draw [-latex ](0.5, 6) -- (1.5, 5);
    \draw [-latex ](0.5, 6) -- (1.5, 6);
    
 \node[draw, text width = 4cm] at (4, 7) {$(c^1)^{\intercal}a(W^1Pu+b^1)$};
  \node[draw, text width = 4cm] at (4, 5) {$(c^K)^{\intercal}a(W^KPu+b^K)$};
 \node[text width=5cm] at (4.4, 6) {$... ...$ \quad $... ...$ \quad $... ...$};

 \draw [-latex ](6.5, 7) -- (7.5, 6);
  \draw [-latex ](6.5, 5) -- (7.5, 6);
    \draw [-latex ](6.5, 6) -- (7.5, 6);

     \node[draw, text width = 4cm] at (10, 6) {Concatenate output to assemble $h \in\mathbb{R}^K$};

\end{tikzpicture}}
\caption{BelNet structure. Projection nets are $C$ independent fully connected neural network with weights and bias $W_1^{i}\in\mathbb{R}^{N_1\times N}$, $W_2^{i}\in\mathbb{R}^{C\times N_1}$ and $b_1^i\in\mathbb{R}^{N_1}$. Construction net is a fully connected neural network with weights and bias $Q\in\mathbb{R}^{K\times d}$ and $b_x\in\mathbb{R}^K$. Here $Q = [q^1, q^2, ..., q^K]$, where $q^i\in\mathbb{R}^d$. Nonlinear nets are $K$ indepedent neural networks. Specifically, 
$c^i\in\mathbb{R}^I$, and $W^i\in\mathbb{R}^{I\times C}$ and $b^i\in\mathbb{R}^{I}$. In addition, $a_x, \sigma, a$ are activation functions.}

\label{fig_network_general}
\end{figure}

\section{Main Results}

\label{sec_proof}

In this section, we prove the universal approximation theorem of BelNet in the sense of operators.
\begin{definition}
    If a function $g:\mathbb{R}\rightarrow \mathbb{R}$ (continuous or discontinuous) satisfies that all linear combinations $\Sigma_{i = 1}^Nc_ig(\lambda_i x+\theta_i)$ are dense in $C[a, b]$, where $c_i, \lambda_i, \theta_i\in\mathbb{R}$, then $g$ is called a \textit{Tauber-Wiener (TW) function}.
\end{definition}
Theorem 3 from \cite{chen1995universal} proves the universal approximation for functions. 
Unlike the universal approximation theorems from \cite{cybenko1989approximation, barron1993universal,jones1992simple}, the approximation coefficients $c_i(f)$ is a functional which depends on the input function $f$.
\begin{lemma}[Theorem 3 from \cite{chen1995universal}]
    Suppose $H\subset\mathbb{R}^{d}$ is compact, $V\subset C(H)$ is also compact, and $g\in TW$. Let $f\in V$ and for any $\epsilon>0$, there exists an integer $K>0$ independent of $f$, and continuous linear functionals $c_i$ on $V$ such that
    \begin{align*}
        \left|f(y) - \sum_{k = 1}^Kc_i(f)g(w_k\cdot y+b_k) \right|<\epsilon,
\end{align*}
for all $y\in H$ and $f\in V$.
\label{chen_theorem3}
\end{lemma}

The following two topological lemmata are used to construct the input function approximation $u_k$ as detailed in Equation \ref{chen_u_k}.

\begin{lemma}[Lemma 5 from \cite{chen1995universal}]
    Let $Y$ be a Banach space and $H\subset Y$, then $H$ is compact if and only if the following two statements are true: 
    \begin{enumerate}
        \item $H$ is closed.
        \item For any $\eta>0$, there is a $\eta-$net $N(\eta) = \{y_1, ..., y_{m(\delta)}\}$, i.e., for any $y\in H$, there is $y_k\in N(\eta)$ such that $\|y-y_k\|<\eta$.
    \end{enumerate}
\end{lemma}
We recall some properties of compact subsets of continuous functions.
\begin{lemma}[Lemma 6 from \cite{chen1995universal}]
    $V\subset C(H)$ is compact set in $C(H)$, then it is uniformly bounded and equicontinuous, i.e., 
    \begin{enumerate}
        \item There is a constant $A>0$ such that $\|u(y)\|_{C(H)}\leq A$ for all $u\in V$.
        \item For all $\epsilon>0$, there exists $\delta>0$ such that $|u(y') - u(y'')|<\epsilon$ for all $u\in V$ provided $\|y'-y''\|<\delta$.
    \end{enumerate}
    \label{chen_lemma6}
\end{lemma}

\begin{remark}
    Let $f$ be a continuous functional on $V$.
Pick a sequence $\epsilon_1>\epsilon_2>...>\epsilon_n\rightarrow 0$, there exists another sequence $\delta_1>\delta_2>...>\delta_n>0$, such that,
$|f(u)-f(v)|<\epsilon_k$, for all $|u-v|<\delta_k$.
By Lemma \ref{chen_lemma6}, there exists a sequence $\eta_1>\eta_2>...>\eta_n\rightarrow 0 $ such that
$|u(y') - u(y'')|<\delta_k$ for all $\|y'-y''\|<\eta_k$ and $u\in V$.
\label{remark_eta_net}
\end{remark}

We can find a sequence $\{z_i\}_{i = 1}^{\infty}\subset H$ and a sequence $m(\eta_1)<m(\eta_2)<...<m(\eta_n)$ such that the first $m(\eta_k)$ elements $N(\eta_k) = \{z_1, ..., z_{m(\eta_k)}\}$ is a $\eta_k$-net of $H$.
For each $\eta_k$-net, and $z_j\in N(\eta_k)$, define a function,
\begin{align*}
    T^{*}_{k, j}(y) = 
    \begin{cases}
        1-\frac{\|y-z_j\|_{H}}{\eta_k}, \|y - z_j\|_{H} \leq \eta_k,\\
        0, \text{ otherwise },
    \end{cases}
\end{align*}
where $y\in H$. Next, we define,
\begin{align}
    T_{k, j}(y) = \frac{T^*_{k, j}(y)}{\sum_{j = 1}^{m(\eta_k)}T^*_{k, j}(y)},
    \label{t_kj}
\end{align}
and a matrix $T^k\in\mathbb{R}^{m(\eta_k)\times m(\eta_k)}$, where the $(i,j)^\text{th}$-entry of $T^k$ is $T_{k, i}(z_j)$.
For any $u\in V$, we define a function,
\begin{align}
    u_{k}(y) = \sum_{j = 1}^{m(\eta_k)}u(z_j)T_{k, j}(y),
    \label{chen_u_k}
\end{align}
and set $\hat{u}^k_z = [u(z_1), ..., u(z_k)]^{\intercal}$. Furthermore, we define $V_k = \{u_k: u\in V\}$ and $\Tilde{V}  = V\cup \left(\cup{k = 1}^\infty V_k\right)$. The next lemma establishes the approximation of $u$ by $u_k$.

\begin{lemma}[Lemma 7 from \cite{chen1995universal}]
    For any $u\in V=C(K_1)$, and $\delta_k>0$, there exists a $\eta_k$-net $N(\eta_k)\subset K_1$, and $u_k$ defined as in equation ($\ref{chen_u_k}$) such that,
    \begin{align*}
        \|u - u_k\|_{C(K_1)}<\delta_k.
    \end{align*}
    \label{chen_lemma7}
\end{lemma}
The next lemma establishes the universal approximation to continuous functional using a two layer networks.
\begin{lemma}[Theorem 4 from \cite{chen1995universal}]
Suppose that $g\in TW$, $Y$ is a Banach space, $K_1\subset Y$ is compact, and $V\subset C(K_1)$ is also compact.
    Let $f$ be a continuous functional on $V$. For any $\epsilon>0$, there exist integers $I, C>0$, weight and bias $W\in\mathbb{R}^{I\times C}$ and $c\in\mathbb{R}^I$, $b\in\mathbb{R}^I$ and $\hat{z} = [z_1, ...., z_C]$ with $z_i\in K_1$ such that,
    \begin{align*}
        \left|f(u) - c^{\intercal}g(W\hat{u}+b) \right| <\epsilon,
    \end{align*}
    for all $u\in V$, and $\hat{u}_z = [u(z_1), ..., u(z_C)]^{\intercal}$.
    \label{chen_theorem4}
\end{lemma}

\begin{remark}
The proof appears in Theorem 4 from \cite{chen1995universal}, and we discuss an important remark regarding the theorem.
The sensors $\{z_i\}_{i = 1}^C$ are the evaluation points for the input function space $V$. They form an $\eta_k-$net of $K_1$, i.e., 
$\{z_i\}_{i = 1}^C = N(\eta_k) = \{z_1, ..., z_{m(\eta_k)}\}$, where we denote $C = m(\eta_k)$. The sensors $\{z_i\}_{i = 1}^C$, constant $C$ and the $\eta_k-$net are determined as follows. For any $\epsilon>0$, choose $m(\eta_k)$ large enough such that $\|u - u_{k}\|_{C(Y)}<\delta_k$ implies $|\Tilde{f}(u) - \Tilde{f}(u_{k})|<\epsilon/2$. 
Here $\Tilde{f}\in \Tilde{V}$ is the extension of $f$ by the Tietze Extension Theorem, i.e.,
\begin{align*}
    f(w)  = \Tilde{f}(w), \forall w\in V.
\end{align*}
One can then define $\eta_k$, and $\eta_k-$net (the sensors $\{z_i\}_{i = 1}^C$) as in Remark \ref{remark_eta_net}.
We will use the sensors $\{z_i\}_{i = 1}^C$ from \cite{chen1995universal} to establish the universal approximation theorem for BelNet.
\label{remark1}
\end{remark}

To prove the universal approximation theorem of BelNet, the key is to show there is a neural network that can map the function values at arbitrary sensors to $\hat{u}_z$. In Lemma~\ref{key_lemma}, we show a strategy for selecting a set of appropriate sensors and then prove the existence of the neural network.
\begin{lemma}
    Let $\hat{z}$ and $C$ be defined from Lemma \ref{chen_theorem4}. For any $\epsilon_u>0$, there exist \textcolor{black}{integers} $N, I>0$, $K_y\subset K_1^N$, and neural networks $\mathcal{N}^i: K_y\rightarrow \mathbb{R}^{C}$,
    \begin{align*}
        \mathcal{N}^i(\hat{y}) = W_2^ia(W_1^i\hat{y} + b_1^i), \quad i\in[N]
    \end{align*}
    and $\mathcal{N}: K_y\rightarrow\mathbb{R}^{C\times N }$ defined as $\mathcal{N}(\hat{y}) = [\mathcal{N}^1(\hat{y}), ..., \mathcal{N}^{N}(\hat{y})]$, such that
    \begin{align*}
        \|\hat{u}_z - \mathcal{N}(\hat{y})u(\hat{y})\|_F<\epsilon,
    \end{align*}
    where $W_1^i\in \mathbb{R}^{I\times N}$, $W_2^i\in \mathbb{R}^{C\times I}$, and for any $\hat{y} = [y_1, ..., y_N]^{\intercal}\in K_y$.
    \label{key_lemma}
\end{lemma}

\begin{proof}

For any $\delta >0$, by Lemma \ref{chen_lemma7}, there is a sufficiently large integer $C_\delta$ such that that 
$\|u - u_k\|_{C(K_1)}<\delta$. Here $u_k(y) = \sum\limits_{j = 1}^{m(\eta_k)}u(r_j)T_{k, j}(y)$ is defined as (\ref{chen_u_k}) and $m(\eta_k) = C_{\delta}$. Moreover we denote $\hat{r} = [r_1, ..., r_{C_{\delta}}]^\intercal$.

For any $N>0$ and $\hat{y}=[y_1,\dots,y_N]^{\intercal}\in K_1^N$ we can define two continuous operators $T_y: K_1^N\rightarrow \mathbb{R}^{N\times C_{\delta}}$ and $T_z:  K_1^C\rightarrow \mathbb{R}^{C\times C_{\delta}}$ as
\begin{align*}
T_y(\hat{y}) = 
    \begin{pmatrix}
T_{k, 1}(y_1) & ... & T_{k, C_{\delta}}(y_1)\\
... & ... & ...\\
T_{k, 1}(y_N) & ... & T_{k, C_{\delta}}(y_N)
\end{pmatrix}, \quad
T_z(\hat{z}) = 
    \begin{pmatrix}
T_{k, 1}(z_1) & ... & T_{k, C_{\delta}}(z_1)\\
... & ... & ...\\
T_{k, 1}(z_C) & ... & T_{k, C_{\delta}}(z_C)
\end{pmatrix},
\end{align*}
where $T_{k, j}$ is defined in Equation (\ref{t_kj}). For any fixed $\epsilon_u$, $\delta$, and $N$, we want to construct a subset $K_y\subset K_1^N$, a continuous $v: K_y\rightarrow \mathbb{R}^{C\times N}$, such that,
 \begin{align}
     v(\hat{y})T_{y}( \hat{y} ) = T_{z}(\hat{z}),
     \label{vT_y_T_z_transformation}
 \end{align}
Let us define $M(\hat{y}) = T_{y}^{\intercal}(\hat{y})T_{y}(\hat{y})$ and set
\begin{align*}
    v(\hat{y}) = T_z(\hat{z})M^{-1}(\hat{y})T^{\intercal}_y(\hat{y})
\end{align*}
for any $\hat{y}\in K_y$. We then define a subset $K_y\subset K_1^N$ as,
\begin{align}
    K_y = \left\{\hat{y}\in K_1^N, \text{$M(\hat{y})$ is invertible and $\|v(\hat{y})\|\leq \frac{\epsilon_u}{2\sqrt{C\delta^2}}-1$} \right\},
    \label{ky_assumption}
\end{align}
where $\|\cdot\|$ is the matrix operator norm. We remark that, for fixed $\epsilon_u$ and $C$, the set $K_y$ is nonempty when $\delta>0$ is sufficiently small and $N$ is sufficiently large (see Remark \ref{remark_nonempty}).

Denote $C_v = \sup_{\forall u\in V}\|u\|_{V}$, it follows from the universal approximation theorem for functions \cite{cybenko1989approximation} that, for any $\frac{\epsilon_u}{2 \sqrt{N C_v^2}}>0$, there exist neural networks $\mathcal{N}^i$ of the form $W_2^i\sigma(W_1^iy+b_1^i)$ and $\mathcal{N}(y) = [\mathcal{N}^1(y), ..., \mathcal{N}^{N}(y)]$ such that,
\begin{align}
    \|v(y) - \mathcal{N}(y)\|_{C(K_1)}<\frac{\epsilon_u}{2 \sqrt{NC_v^2}}.
    \label{uap_for_v}
\end{align}

For $\hat{y} = [y_1, ..., y_N]^{\intercal}\in K_y$ and $u_k(y) = \sum\limits_{j = 1}^{C_{\delta}}u(r_j)T_{k, j}(y)$, by multiplying both sides of $u_k$ by $v(\hat{y})$, it follows that,
\begin{align}
    v(\hat{y})u_{k}(\hat{y}) = \sum_{j = 1}^{C_{\delta}}u(r_j)v(\hat{y})T_{k, j}(\hat{y})= \sum_{j = 1}^{C_{\delta}}u(r_j)T_{k, j}(\hat{z}_k) = u_k(\hat{z}).
    \label{eqn_vuy_uz}
\end{align}
By equation (\ref{eqn_vuy_uz}) and Cauchy-Schwartz, we have the bound:
        \begin{align*}
        \left\|\mathcal{N}(\hat{y})u(\hat{y})-u(\hat{z}) \right\|_F &= \left\| \big(\mathcal{N}(\hat{y})-v(\hat{y})+v(\hat{y})\big)u(\hat{y}) - u(\hat{z}) \right\|_F\\
        & = \left\| \big(\mathcal{N}(\hat{y})-v(\hat{y})\big)u(\hat{y}) + v(\hat{y})\big(u(\hat{y}) - u_{k}(\hat{y})+u_{k}(\hat{y})\big) - u(\hat{z})  \right\|_F\\
        & \leq \left\| \big(\mathcal{N}(\hat{y})-v(\hat{y})\big)u(\hat{y})  \right \|_F + 
        \left\| v(\hat{y})\big(u(\hat{y})-u_{k}(\hat{y})\big)  \right\|_F + 
        \left\| u_k(\hat{z}) - u(\hat{z})  \right\|_F\\
        &\leq \|\mathcal{N}(\hat{y}) - v(\hat{y})\|\|u(\hat{y})\|_F + (\|v(\hat{y})\|+1)\sqrt{C\delta^2}.
    \end{align*}
    Utilizing (\ref{ky_assumption}) and (\ref{uap_for_v}), the estimation follows.

\end{proof}
\begin{remark}
\label{remark_nonempty}
    We present one example to show $K_y$ is non-empty. Let  $\hat{y} = [\hat{r}, \hat{r}, ... \hat{r}]$, where we repeat $\hat{r}$ $n$ times, and define $T_y(\hat{y})$ by,
    \begin{align*}
        T_y(\hat{y}) = \begin{pmatrix}
            T_r\\
            ...\\
            T_r
        \end{pmatrix},
        \text{ where }
        T_r = T_y(\hat{r}).
    \end{align*}
We have $M = T^{\intercal}_yT_y = nT_r^{\intercal}T_r$. Thus if $T_r$ has full column rank $C_\delta$, then the matrix $M$ is invertible and it follows that
\begin{align*}
    M^{-1}T^{\intercal}_y = \frac{1}{n}[(T_r^{\intercal}T_r)^{-1}T_r^{\intercal}, ..., (T_r^{\intercal}T_r)^{-1}T_r^{\intercal} ].
\end{align*}
We estimate the operator norm of $M^{-1}T^{\intercal}_y$ by studying its largest singular value $\sigma_1$. We have,
$M^{-1}T^{\intercal}_y(M^{-1}T^{\intercal}_y)^{\intercal} = \frac{1}{n}(T_r^\intercal T_r)^{-1}$ which implies that $\|M^{-1}T^{\intercal}_y\| = \sigma_1\leq  \sqrt{\frac{1}{n}\|(T_r^\intercal T_r)^{-1} \|}$. By letting $n$ be large enough, $\|v\| = \|T_z\|\|M^{-1}T^{\intercal}_y\|$ can be sufficiently small, and thus $K_y$ is non-empty.
\end{remark}

\begin{remark}
    $M(y) = T_{y}^{\intercal}(y)T_{y}(y)$, this implies that $rank(M) = rank(T_y)\leq \min(C_{\delta}, N)$. Since $M\in\mathbb{R}^{C_{\delta}\times C_{\delta}}$, $M$ is singular if $N<C_{\delta}$. 
\end{remark}


\begin{remark}
    $\mathcal{N}$ is the projection net in Figure \ref{fig_network_general}. $\mathcal{N}u$ is the projection coefficients of $u$ onto a set of functions (``basis'') implicitly learned. 
\end{remark}

\begin{theorem}[\textbf{Universal Approximation Theorem for BelNet}]
Suppose that $a\in TW$, $Y$ is a Banach space, $K_1\subset Y$, $K_2\subset \mathbb{R}$ are all compact.
$V\subset C(K_1)$ is compact and $G: V\rightarrow C(K_2)$ is continuous and nonlinear.
For any $\epsilon>0$, there exist integers $N,C,K,I$, weights and biases $W_x^k\in\mathbb{R}^{d}$, $b_x^k\in\mathbb{R}$, $W^k\in\mathbb{R}^{I\times C}$, $b^k\in\mathbb{R}^{I}$, $c^k\in\mathbb{R}^{I}$, subset of sensors $K_y\subset K^N_1$ and a trainable network $\mathcal{N}: K_y\rightarrow \mathbb{R}^{C\times N}$ specified in Lemma \ref{key_lemma}, where $K_y$ satisfies Equation (\ref{ky_assumption}). 
Then the following inequality holds
\begin{align*}
    \left|
G\left(u\right)(x) - \sum_{k = 1}^K a(W_x^k\cdot x+ b^k_x) (c^k)^{\intercal}a\big(W^k\mathcal{N}(\hat{y})u(\hat{y}) + b^k \big)
    \right|<\epsilon,
\end{align*}
for all $x\in K_2$, $\hat{y} = [y_1, y_2, ..., y_N]^{\intercal}\in K_y$, and $u\in V$.
\end{theorem}

\begin{proof}
Since $G$ is continuous and $V\subset C(K_1)$ is compact, the range $G(V)$ is also compact in $C(K_2)$.
By Lemma \ref{chen_theorem3},
for any $\epsilon>0$, there exist a positive integer $K$ and linear continuous functional $L_k$, $W_x^k\in\mathbb{R}^{d}$, $b_x^k\in\mathbb{R}$ such that
\begin{align*}
    \left| G\left(u\right)(x) - \sum_{k = 1}^K L_k\big(G(u)\big) a(W_x^k\cdot x+ b^k_x) \right|<\cfrac{\epsilon}{3},
\end{align*}
for all $x\in K_2$ and $u\in V$.
By Lemma \ref{chen_theorem4}, for all $k$, there exist integers $C, I$, $\hat{z} = [z_1, ..., z_C]^{\intercal}$ with $z_i\in K$,
$c^k\in\mathbb{R}^{I}$, $W^k\in\mathbb{R}^{I\times C}$, $b^k\in\mathbb{R}^{I}$ such that 
\begin{align*}
    |L_k\left(G(u)\right) - (c^k)^{\intercal}a(W^k\hat{u}_z + b)|<\frac{\epsilon}{3C_u K},
\end{align*}
where $\hat{u}_z = [u(z_1), ..., u(z_C)]^{\intercal}$ and $C_u=\max\limits_{k,x\in K_2}{a(W^k_x\cdot x+b^k_x)}$. Therefore, we obtain an approximation to $G(u)(x)$ as in \cite{chen1995universal} defined by
\begin{align}
    \mathcal{G}\left(u(\hat{z})\right)(x) = \sum_{k = 1}^K a(W_x^k\cdot x+ b^k_x)(c^k)^{\intercal}a(W^k u(\hat{z})+ b^k)
    \label{eqn_final_eqn2}
\end{align}
with $\Big|\mathcal{G}\left(u(\hat{z})\right)(x)-G(u)(x)\Big|<\cfrac{2\epsilon}{3}$.
Since $a$ is continuous and $K_1^C$ is compact, we can define uniformly continuous \textcolor{black}{functions} $a_k:K_1^C\rightarrow \mathbb{R}$:
\[
a_k(\hat{u}) = a(W^k\hat{u}+b^k).
\]
Thus, there is an $\epsilon_u>0$, such that 
\begin{align}
    |a_k(\hat{u}')-a_k(\hat{u}'')|<\cfrac{\epsilon}{3KLC_u}
    \label{eqn_final_eqn0}
\end{align}
for all $\|\hat{u}'-\hat{u}''\|_F < \epsilon_u$ where $L=\max\limits_k{\|c_k\|_{l^1}}$.

By Lemma \ref{key_lemma}, there exists $N,\mathcal{N}$, and $K_y\subset K^N_1$ such that 
\begin{align}
    \|u(\hat{z})-\mathcal{N}(\hat{y})u(\hat{y})\|_F<\epsilon_u\; \quad \text{for any }y\in K_y.
    \label{eqn_final_approx}
\end{align}
Letting $\hat{y} = [y_1, ..., y_N]\in K^N_1$, the difference is bounded by
\begin{align*}
    &\left| G\left(u\right)(x) - \sum_{k = 1}^K a(W_x^k\cdot x+ b^k_x) (c^k)^{\intercal}a\big(W^k\mathcal{N}(\hat{y})u(\hat{y}) + b^k \big) \right|\\
    & \leq  \underbrace{|G\left(u\right)(x) - \mathcal{G}(u(\hat{z}))(x)|}_{\mathcal{E}_1}\\
    &+ \underbrace{\left|\mathcal{G}(u(\hat{z}))(x)
    -\sum_{k = 1}^K a(W_x^k\cdot x+ b^k_x) (c^k)^{\intercal}a\big(W^k\mathcal{N}(\hat{y})u(\hat{y}) + b^k \big) \right|}_{\mathcal{E}_2}.
\end{align*}
By equations \ref{eqn_final_eqn0}, \ref{eqn_final_eqn2} and \ref{eqn_final_approx}, the second term $\mathcal{E}_2$ is controlled by: 
\begin{align}
    \mathcal{E}_2 &= \left| \sum_{k = 1}^K  a(W_x^k\cdot x+ b^k_x) (c^k)^{\intercal} \left(a_k(\hat{z})-a_k(\mathcal{N}(\hat{y})u(\hat{y}))\right) \right| <\frac{\epsilon}{3},
\end{align}
and the total approximation follows accordingly.
\end{proof}

\section{Numerical Experiments}

\label{sec_numerical}

We apply our approach to a nonlinear scalar PDE and multiscale PDE problems. Specifically, we first test the proposed BelNet extension on the viscous Burgers' equation. Then we show that BelNet can be used to address some of the difficulties associated with learning multiscale operators. \textcolor{black}{The code and examples will be available when the work is published.}

\subsection{Parametric Viscous Burgers' Equation}
Consider the viscous Burgers' equation with periodic boundary conditions: 
\begin{align*}
    &\frac{\partial u_s}{\partial t} + \frac{1}{2}\frac{\partial (u^2_s)}{\partial x} = \alpha \frac{\partial^2 u_s}{\partial x^2},\hspace{0.5em} x\in[0, 2\pi], \hspace{0.2em} t\in[0, 0.3]\\
    &u_s(x, 0) = u^0_s(x),\\
    &u_s(0, t) = u_s(2\pi, t),
\end{align*}
where $u^0_s(x)$ is the initial condition that depends on the parameter $s$ and the viscosity is set to $\alpha = 0.1$.
We consider the operator that maps from the initial condition to the terminal solution at $t = 0.3$.

\textbf{Training Data:} In order to obtain more variability between initial samples for the training phase and to include different levels of steepness in the derivative of the initial data, we generate the initial conditions as follows.  We first compute a short-time solution ($t=0.1$) to Burgers' equation using the periodic boundary conditions, set the viscosity to zero, and use the initial condition $s\sin(x)$ where $s\in [0, 4]$. The solution of the system at $t = 0.1$ is then used as the initial condition $u^0_s$ (resetting time to zero); see the yellow and blue curves in Figure \ref{pic_vburgers_info} as a display. 
\begin{figure}[H]
\centering
\includegraphics[scale = 0.45]{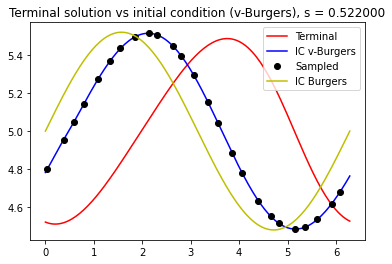}
\includegraphics[scale = 0.45]{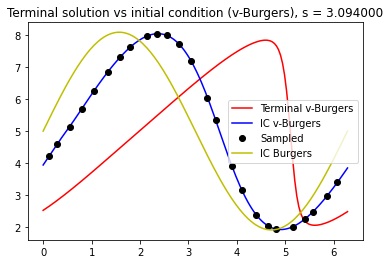}
\caption{Plots of two solutions to the viscous Burgers' equation with our initialization procedure. Note that each example's sampling points  (i.e. the sensors represented by the black dots) for the initial condition differ. The yellow curves are used to generate the initial conditions for the model problem (viscous Burgers' equation). The initial conditions for the viscous Burgers' equation are displayed in blue.}
\label{pic_vburgers_info}
\end{figure}
The mesh for the input data is as follows. Each initial condition (input function) has $25$ sensors, and we used a total of $200$ initial conditions for training. For each initial condition, the true system is evolved up to time $t=0.3$, and a total of $5$ time-stamps are collected (the terminal time is not included). Therefore the space-time mesh contains $25$-by-$5$ total sample locations for each initial condition.

\textbf{Testing and observations:} We train $100$ independent models with the same training dataset, test on the same dataset of $500$ samples and compute the average relative error of 100 predictions. To test the neural operators' ability to forecast future states,  we do not include the solution at the terminal time $t=0.3$ in the training dataset. For testing, we use solutions from $500$ initial conditions and test each neural operator on the solution at the terminal time with a finer mesh of $151$ grid points. We present the relative errors in Table~\ref{table_vburger_err}. We compare BelNet with the vanilla BelNet, which was previously shown to be more accurate than comparable models \cite{zhang2022belnet}. With fewer trainable parameters, listed in Table \ref{table_vburger_err}, BelNet obtains a small prediction error than the vanilla BelNet.

\begin{table}[H]
    \centering
    \begin{tabular}{c|c c}
         Model & Relative Error & Parameter Count \\
        \hline
        vanilla BelNet & $1.42\%$ &  102.93K\\
        \hline
        BelNet & $1.32\%$ & 96.5K
    \end{tabular}
    \caption{Relative errors and trainable parameter counts for viscous Burgers equation. The top row is the vanilla BelNet, while the second row is the BelNet. We perform $100$ independent experiments and present the average relative errors.}
    \label{table_vburger_err}
\end{table}


\subsection{Multiscale Operator Learning}
We test BelNet's performance on the multiscale operator learning problem. In particular, we apply BelNet to improve a coarse-scale (low-accuracy) solution from a multiscale PDE solver. This problem was introduced in \cite{zhang2023homogenization} with the DON framework. 
Let $u_0$ denote a low-accuracy coarse-scale solution of a given PDE; the target is to construct an operator $G$ such that $G(u_0)(\cdot)$ is a fine-scale solution of the PDE. 

To learn the operator, we assume that some observed fine-scale solution data is available.  If we denote an approximation to the input function $u_0$ as $\hat{u}_0$ and $u(x_i)$ as the fine-scale observed solution at $x_i$, the dataset for training can then be denoted as $\{x_i, \hat{u}_0, u(x_i)\}_{i = 1}^{N_p}$.
We can then construct the loss function as,
\begin{align}
    \sum_{i = 1}^{N_p}\|u(x_i) - G_{\theta}(\hat{u}_0)(x_i)\|^2,
\end{align}
where $G_{\theta}$ denotes the neural network with trainable parameter $\theta$.

Since DON is not discretization invariant, i.e., the input function must be discretized in the same way,  $\hat{u}_0$ is then sampled in the same way for all training samples. This limits the potential accuracy as seen in numerical tests. To fix this,  the authors of \cite{zhang2023homogenization} used a localized patch discretization of $u_0$ which was shown to improve the performance of the DON approximation. However, this is theoretically inconsistent with the DON framework.

Since BelNet is discretization-invariant, a local patch discretize of $u_0$ is theoretically consistent. Let us denote $P_i$ as a patch (neighborhood) around at $x_i$, which is the observed solution coordinate. For example, a three-point patch for a $1d$ problem is $P_i = \{x_i-h_{i}^1, x_i, x_i+g_{i}^1\}$, where $h_{i}^1$ and $g_i^1$ are real numbers. The patch is used to discretize the input function $u_0$, i.e., $\hat{u}_i = u_0|_{P_i} = [u_0(x_i-h_i^1), u_0(x_i), u_0(x_i+g_i^1)]^{\intercal}$ is the local discretization of $u_0$. 
To make the problem more challenging, we assume $h_i$ and $g_i$ are different for all $i$.
We present a 2D display of a patch and the candidate sensors' position in Figure \ref{fig_rdm_domon}.

\begin{figure}[H]
    \centering
\begin{tikzpicture}[scale=1]

\draw[step=0.5cm, gray] (-4.5,0) grid (-2., 2);

\foreach \i in {0, ..., 4}
{
\foreach \j in {0, ..., 4}
{
\filldraw[red] (-4+\i/2, \j/2) circle (2pt) node[anchor = west]{};
}
}

\filldraw[black] (-3, 1) circle (3pt) node[anchor=west]{};

\filldraw[yellow] (-4, 0.5) circle (2pt) node[anchor=west]{};

\draw (-4, .5) node[cross = 3pt, blue] {};
\draw (-3.5, .5) node[cross = 3pt, blue] {};
\draw (-4.5, .5) node[cross = 3pt, blue] {};

\draw (-4, 1.) node[cross = 3pt, blue] {};
\draw (-3.5, 1.) node[cross = 3pt, blue] {};
\draw (-4.5, 1.) node[cross = 3pt, blue] {};

\draw (-4, 0) node[cross = 3pt, blue] {};
\draw (-3.5, 0) node[cross = 3pt, blue] {};
\draw (-4.5, 0) node[cross = 3pt, blue] {};

\end{tikzpicture}
    \caption{Plot of a $5\times 5$ patch (red dots) centered at an observation point (black dot).
    To make the problem more challenging, we randomize the sensor position. Specifically, we randomly place a sensor in a neighborhood centered at each red dot.
    Blue crosses are all candidate locations to place sensors for the yellow dot, we uniformly pick one blue 'x' to place one sensor.}
    \label{fig_rdm_domon}
\end{figure}
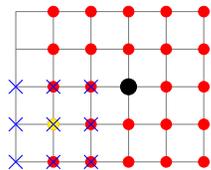

\subsubsection{One dimensional elliptic equation}
We first study a 1D example for which we can obtain an exact homogenized solution $u_0$.
In particular, let us consider the following equation:
\begin{align*}
    &-\frac{d}{dx} \left(\kappa(x/\epsilon)\frac{du}{dx}\right) = f, \quad x\in[0, 1],\\
    & u(0) = u(1) = 0,
\end{align*}
where $\kappa(x) = 0.5\sin(2\pi\frac{x}{\epsilon}) + 0.8 $ and $f(x) = 0.5$.  We plot the multiscale permeability \textcolor{black}{$\kappa(x)$} and the reference solution in Figure (\ref{1d_sprb}).
\begin{figure}[H]
\centering
\includegraphics[scale = 0.5]{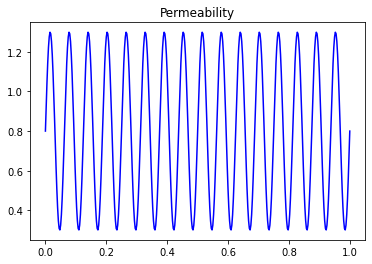}
\includegraphics[scale = 0.5]{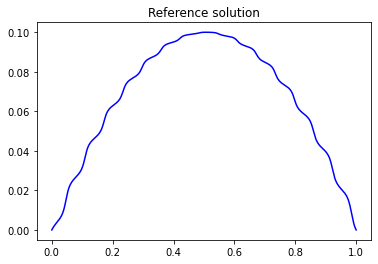}
\caption{1D elliptic. Left: permeability \textcolor{black}{$\kappa$}. Right: reference solution.}
\label{1d_sprb}
\end{figure}
The relative error of the homogenized solution is $0.07\%$; we use exact solutions $u(x_i)$ as the observations, $x_i$ are uniformly distributed, and $N_p = 16$ points are used in total. \textcolor{black}{We use the oversampling trick which employs a patch of coarse-scale solutions to capture the input function (see Figure \ref{fig_rdm_domon}),} and the result is presented in Figure \ref{fig_1d_results}.

\begin{figure}[H]
    \centering
    \includegraphics[scale = 0.5]{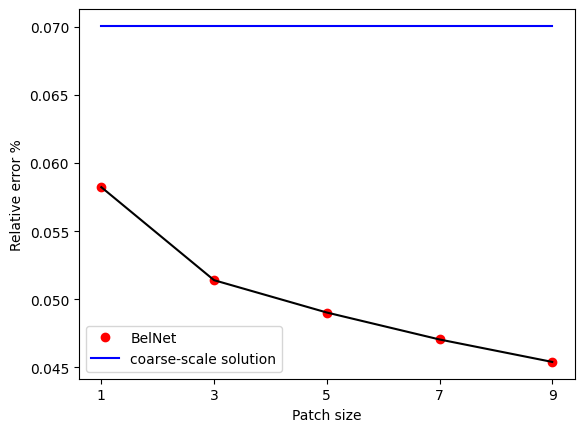}
    \caption{Relative errors with respect to different patch sizes. 
    For each patch size, we perform $100$ independent experiments and present the average relative error. All input functions (low accuracy solution) of each experiment do not share the discretization, and the discretizations for the same input function are not the same in $100$ independent experiments. The larger the patch size, the more flexibility in sampling the input function, and hence is more challenging to train. We do not observe an increase of the relative error with respect to the patch, which also implies that BelNet is discretization-invariant for this example. }
    \label{fig_1d_results}
\end{figure}

\textbf{Settings, Observations, and Comments: } 
We use $N_p = 16$ exact solutions uniformly distributed in the domain to improve $u_0$. It is important to note that our focus does not involve studying the decay of error in relation to the number of observation points, and for a comprehensive investigation, please refer to the work by \cite{zhang2023homogenization}.

As the size of the patch expands, there is increased flexibility in terms of sensor placement for sampling the input function. Specifically, we consider the placement of $1$, $3$, $5$, $7$, and $9$ sensors, respectively. The locations of the sensors are randomized for each local patch, denoted as $h_i^j$ and $g_i^j$, which vary across different samples of $u_0(x_i)$. Refer to Figure \ref{fig_rdm_domon} for a display of this randomization process. With an increasing number of sensors, the various possible discretizations of each sample $u_0$ become more numerous, resulting in a more challenging training scenario when dealing with larger patches. However, the results presented in Figure \ref{fig_1d_results} demonstrate that BelNet exhibits discretization invariance, as the accuracy remains unaffected and improves with patch size.

\subsubsection{2D elliptic equation with one fast variable}
We consider the following 2D elliptic equation:
\begin{align}
    &-\nabla\cdot \left(\kappa(x/\epsilon)\nabla u\right) = f, x\in\Omega = [0, 1]^2,\\
    & u(x) = 0, x\in \partial\Omega,
    \label{eqn_elliptic}
\end{align}
where $\kappa(x/\epsilon) = 2 + \sin(2\pi x/\epsilon)\cos(2\pi y/\epsilon)$ and $\epsilon = \frac{1}{8}$. We display $\kappa(x)$ in Figure (\ref{2d_kappa}).
\begin{figure}[H]
\centering
\includegraphics[scale = 0.5]{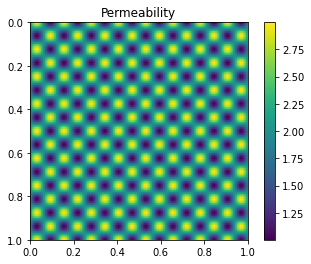}
\includegraphics[scale = 0.5]{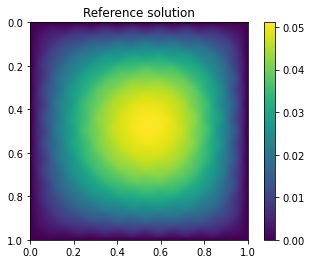}
\caption{2D elliptic with $1$ fast variable. Left: permeability $\kappa$. Right: reference solution.}
\label{2d_kappa}
\end{figure}

We measure the input function (the low-accuracy solution) by sampling it in a neighborhood around the observation sensor. 
In order to obtain the low-accuracy solution, we employ mesh-dependent solvers that define all solutions on grid points. As the neighborhood around the grid point sensor expands, there is an increase in the degrees-of-freedom available for sampling the input function. It is important to highlight that the utilization of different discretizations for the input functions poses difficulties during training. Therefore, we require a discretization-invariant tool such as BelNet to address this issue.

\textbf{Settings and comments on the results:}
We increase the patch size, perform $100$ independent experiments for each patch size and compute the average relative error. In each experiment, the input function value $u_0$, representing the low-accuracy solution, was measured by sampling random points within the patch (neighborhood) surrounding the observation point $x_i$.

The results are displayed in Figure \ref{2d_homog_pacth}. When the patch size is $1$, only a single point is used to sample $u_0$, rendering it insufficient for an accurate approximation. As the patch size increases, training becomes more challenging due to the increased freedom in sensor placement for sampling $u_0$. However, the approximation error decreases (in trend) indicating that BelNet can effectively handle problems with varying input function meshes.

\begin{figure}[H]
    \centering
\includegraphics[scale = 0.5]{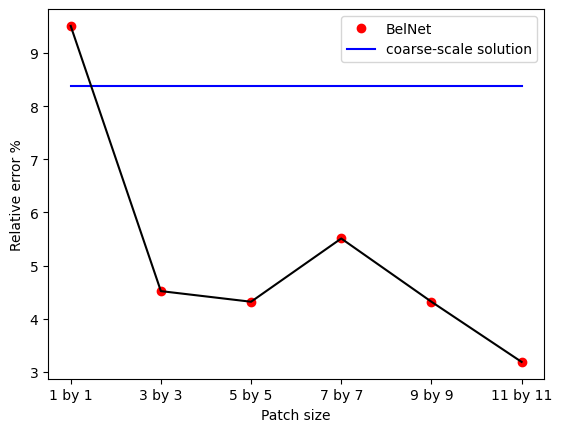}
    \caption{Relative errors with respect to different patch sizes. 
    For each patch size, we conducted $100$ independent experiments and calculated the average relative error.
It is important to note that the input functions (representing low-accuracy solutions) used in the experiments do not share the same discretization. Furthermore, the discretization of each input function varies across the $100$ independent experiments.}
    \label{2d_homog_pacth}
\end{figure}


\subsubsection{2D elliptic multiscale PDE}
We consider the same equation (\ref{eqn_elliptic}) but with different permeability $\kappa$:
\begin{align*}
    \kappa(x, y) = 1 + \frac{\sin(2\pi\frac{x}{\epsilon_0})\cos(2\pi\frac{y}{\epsilon_1}) }{2+\cos(2\pi\frac{x}{\epsilon_2})\sin(2\pi\frac{y}{\epsilon_3})} +
    \frac{\sin(2\pi\frac{x}{\epsilon_4})\cos(2\pi\frac{y}{\epsilon_5})}{2+\cos(2\pi\frac{x}{\epsilon_6})\sin(2\pi\frac{y}{\epsilon_7})},
\end{align*}
where $\epsilon_0 = \frac{1}{5}$, $\epsilon_1 = \frac{1}{4}$, $\epsilon_2 = \frac{1}{25}$, $\epsilon_3 = \frac{1}{16}$, $\epsilon_4 = \frac{1}{16}$, $\epsilon_5 = \frac{1}{32}$, $\epsilon_6 = \frac{1}{3}$, and $\epsilon_7 = \frac{1}{9}$. We plot the permeability and the solution in Figure (\ref{non_sprb}).
\begin{figure}[H]
\centering
\includegraphics[scale = 0.5]{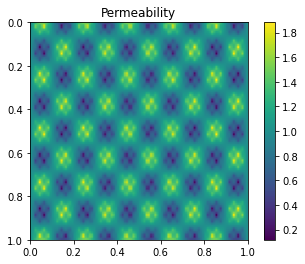}
\includegraphics[scale = 0.5]{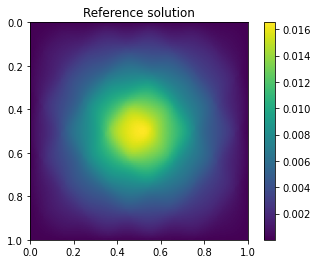}
\caption{2D elliptic multiscale PDE. Left: permeability $\kappa$. Right: reference solution.}
\label{non_sprb}
\end{figure}

We obtain the coarse-scale solution by the multiscale finite element method with one local basis \cite{efendiev2013generalized,chung2016adaptive,chung2018constraint,chetverushkin2021computational}.
We conduct six sets of experiments with patch sizes $1\times 1$, $3\times 3$, and $5\times 5$, $7\times 7$, $9\times 9$ and $11\times 11$. We train $100$ models for each set of experiments and compute the average relative errors of the last $100$ epochs of $100$ models. The results are shown in Figure \ref{2d_homog_ms_pacth}.

\begin{figure}[H]
    \centering
\includegraphics[scale = 0.5]{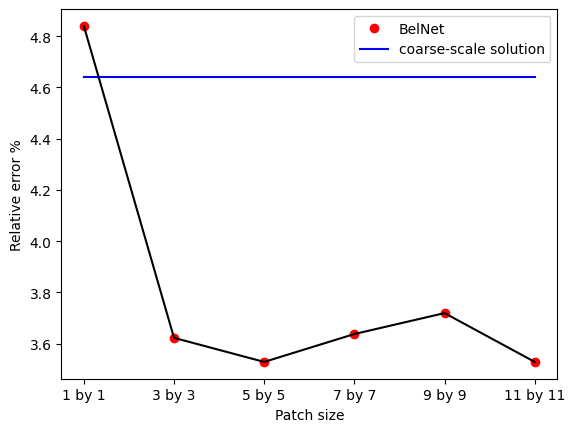}
    \caption{Relative errors with respect to different patch sizes. 
    For each patch size, we perform $100$ independent experiments and present the average relative error.
    It should be noted that the input functions (representing the low accuracy solution) used in the experiments do not share the same discretization. Additionally, the discretization of each input function varies among the $100$ independent experiments.
 Even as the patch size expands, BelNet maintains stability. }
    \label{2d_homog_ms_pacth}
\end{figure}


\textbf{Settings, Observations, and Comments: } We used a set of $16$ input function samples, denoted as $u(x_i)$, which are uniformly distributed to improve the initial function $u_0$. When the patch size is $1$, all input functions are sampled as $u_0(x_i)$. However, this sampling strategy does not yield a satisfactory approximation to $u_0$, resulting in (slightly more) inaccurate prediction as compared to a coarse-scale solution. As the patch size increases from $1\times 1$ to $11\times 11$, we incorporate a larger number of sensors. Specifically, the number of sensors used is $1$, $9$, $25$, $49$, $81$, and $121$, respectively. Despite the numerical challenge of using more sensors which are non-overlapping, the prediction accuracy does not deteriorate significantly. The error has a slight increase between $5 \times 5$ and $9\times 9$ patches which may indicate a saturation of the error within this patch size window (since the error remain around $3.7\%$). The relative error decreases again as the patch size increases.

\section{Conclusion}
We generalize the vanilla BelNet architecture proposed in \cite{zhang2022belnet} by adding a trainable nonlinear layer in the network. We prove the universal approximation theorem of BelNet in the sense of operators, extending the results of \cite{chen1993approximations,chen1995universal, lu2021learning}.
In particular, we show that BelNet can be viewed as a discretization-invariant extension of the operator networks in \cite{chen1995universal,chen1995universal} which allows for several new applications, particularly to multiscale PDE. In particular, the discretization-invariance property allows for the input functions to be observed at different sensor locations which is often the case for applications where the sensor locations move in time or where fluctuations in data acquisition occurs. For multiscale problems, the randomization of patch location and neighboring points necessitates the use of discretization-invariant learning. We test the performance on high-contrast and multiscale parametric PDE. Our experiments show that BelNet typically obtains about $1.2\times$ to $2\times$ improvement in relative error over the course-scale solution without needing to fully resolve the multiscale problem. \textcolor{black}{Lastly, it is worth noting that part of the theoretical analysis shows that the network obtains a (trained) reduced order model and projection. This is a useful result for assessing the contributions of individual subnetworks within the full operator learning architecture, i.e. peering into the black-box of deep networks for PDE.}

\section{Acknowledgement}
Z. Zhang was supported in part by AFOSR MURI FA9550-21-1-0084. H. Schaeffer was supported in part by AFOSR MURI FA9550-21-1-0084 and an NSF CAREER Award DMS-2331100.

\bibliographystyle{abbrv}
\bibliography{references}
\end{document}